\documentclass[12pt]{article}

\usepackage{amsmath}
\usepackage{amsfonts}
\usepackage{latexsym}
\usepackage{amssymb}
\usepackage{enumerate}
\makeatletter

\newtheorem{theorem}{Theorem}[section]

\newtheorem{lemma}[theorem]{Lemma}

\newcommand{\qed}{\protect\nolinebreak{\hfill$\Box$}}
\newtheorem{unnumber}{}

\newtheorem{prepf}[unnumber]{Proof.}
\newenvironment{proof}{\prepf\rm}\qed{\endprepf}
\def\0{\mathbf 0}

\def\cC{\mathcal C}
\def\cD{\mathcal D}
\def\cE{\mathcal E}

\def\cG{\mathcal G}

\def\cN{\mathcal N}

\def\cV{\mathcal V}

\def\cU{\mathcal U}

\def\fE{\mathfrak E}
\def\fI{\mathfrak I}
\def\PG{{\rm PG}}
\def\AG{{\rm AG}}
\def\GF{{\rm GF}}

\def\PGL{{\rm PGL}}

\def\det{{\rm det}}

\def\GF{{\rm GF}}

\newcommand\comment[1]{}

\title{Unitals of $\PG(2,q^2)$ containing conics}

\author{N. Durante and A. Siciliano}

\begin{document}


\maketitle

\begin{abstract}
A {\em unital} in $\PG(2,q^2)$ is a set $\cU$ of $q^3+1$ points such that each
line meets $\cU$ in 1 or $q+1$ points. The well known example is the classical
unital consisting of all absolute points of a unitary polarity
of $\PG(2,q^2)$. Unitals other than the classical one also exist in
$\PG(2,q^2)$ for every $q>2$. Actually, all known unitals are of
Buekenhout--Metz type, see \cite{b,m}, and they can be obtained by a
construction due to Buekenhout \cite{b}. The unitals constructed by
Baker--Ebert \cite{be}, and independently by Hirschfeld--Sz\H{o}nyi \cite{hs},
are the union of $q$ conics. Our Theorem \ref{th_2} shows that this geometric
property characterizes the Baker--Ebert--Hirschfeld--Sz\H{o}nyi unitals.
\end{abstract}
 \section{Introduction}

Let $\PG(2,q)$, $q=p^h$, $p$ a prime, denote the Desarguesian projective plane
of order $q$. A {\em maximal arc} of {\em degree} $n$ is a set of points of
$\PG(2,q)$ meeting every line in either $0$ or $n\le q$ points. For example, a
point or the complement of a line are maximal arcs; these are called {\em
trivial} maximal arcs. In \cite{bbm} it was proved that no non-trivial
maximal arc exists in $\PG(2,q)$, with $q$ odd. Instead, in \cite{d},
Denniston constructed maximal arcs in $\PG(2,q)$, $q$ even, each of which 
is the union of
irreducible conics from a partial pencil plus their common nucleus \cite{d}.

In \cite{t1,t2} J.A. Thas constructed two classes of maximal arcs of
$\PG(2,q)$, $q$ even. In \cite{hp} it was proved that some of the maximal arcs
in the first class as well as all maximal arcs of the second class are of
Denniston type. Many years later, Mathon studied the following problem:

 {\em Do there exist other maximal arcs in $\PG(2,q)$, each of which is
  the union of conics plus their common nucleus?}

 In his paper \cite{m} he gave a positive answer by constructing the 
 {\em Mathon maximal arcs}.

In this paper we deal with a similar problem about unitals of $\PG(2,q^2)$.

A {\em unital} in $\PG(2,q^2)$ is a set $\cU$ of $q^3+1$ points such that each
line meets $\cU$ in 1 or $q+1$ points. A line of $\PG(2,q^2)$ is a {\em
tangent} or {\em secant} line to $\cU$ according if it contains 1 or $q+1$
points of $\cU$. Through each point of $\cU$, there is exactly one tangent and
$q^2$ secants to $\cU$, while through each point not in $\cU$, there are $q+1$
tangents and $q^2-q$ secant lines.

 An example of a unital is given by the set of absolute points of a
 non-degenerate unitary polarity of $\PG(2,q^2)$. This is a {\em
 classical} or {\em Hermitian unital}.

In \cite{b,m} Buekenhout and Metz constructed non-classical unitals by using
the Andr\`e/Bruck--Bose representation of $\PG(2,q^2)$ in $\PG(4,q)$ for
$q>2$. These unitals are {\em Buekenhout--Metz unitals}. In \cite{be,hs} a
nice geometric description in $\PG(2,q^2)$, $q$ odd, was given for some of
these unitals. For $a\in\GF(q^2)$, consider the conic $\cC_a$ with equation
$2yz-x^2+az^2=0$. The set $\{\cC_a: a\in \GF(q^2)\}$ is a hyperosculating
pencil with base point $(0,1,0)$. Let $t$ be a fixed non-square of $\GF(q^2)$.
Then the set 
\[ 
\cU=\bigcup_{a\in t\GF(q)}{\cC_a} 
\] 
turns out to be a
Buekenhout--Metz unital that we call of {\em
Baker--Ebert--Hirschfeld--Sz\H{o}nyi type} or {\em BEHS-type} for short. The
following question arises:

{\em Do there exist other unitals of $\PG(2,q^2)$ which are unions of conics?}

The answer is negative.

\begin{theorem}\label{th_2}
Let $\cU$ be a unital of $\PG(2,q^2)$ and suppose that $\cU$ is a union of
conics. Then $q$ is odd and $\cU$ is a Buekenhout--Metz unital of BEHS-type.
\end{theorem}

\section{Proof of Theorem \ref{th_2}}

Let $\cU$ be an unital of $PG(2,q^2)$ and let $\cC$ be an irreducible conic
contained in $\cU$. For every point $P$ of $\cC$, the tangent at $P$ to $\cC$
coincide with the tangent at $P$ to $\cU$ .

For $q$ even, the tangents to $\cC$ all contain a common
point $\cN$, the nucleus of $\cC$, \cite[Chapter 7]{h}. Thus there would be $q^2+1$
tangents to $\cU$ on $\cN$, a contradiction. Hence, we conclude that if $\cU$
contains an irreducible conic, then $q$ must be odd.

{}From now on, $q$ is an odd prime power and $\cU$ a union of irreducible
conics.

In \cite{pr} Penttila and Royle gave a complete classification of
two-intersection sets in the projective planes of order $9$. From this
classification, the Buekenhout--Metz unitals of BEHS-type are the only unitals
in $\PG(2,9)$ containing conics. Thus we may assume $q>3$.

In $\PG(2,q)$, equipped with the homogeneous coordinates $(x,y,z)$, any conic
$\cC$ is defined by the equation
\begin{equation}\label{eq_1}
f(x,y,z)=a_{11}x^2 +a_{22}y^2 + a_{33}z^2+ 2a_{12}xy + 2a_{13}xz 
+ 2a_{23}yz  = 0
\end{equation}
and the associated symmetric matrix is
\[
A(\cC)=\begin{pmatrix}
a_{11} 	& a_{12}& a_{13}	\\
a_{12}	& a_{22}& a_{23}	\\
a_{13} 	& a_{23}& a_{33}	
\end{pmatrix}.
\]

The {\em rank} of $\cC$ is the rank of the matrix $A(\cC)$. Conics of rank 3
are said to be {\em irreducible} or {\em non-singular}. Singular conics are of
two types: a pair of distinct lines (when the associated matrix has rank 2)
and a repeated line (when the associated matrix has rank 1).

If $\cC$ is irreducible, the points that are not in $\cC$ split in two sets:
the set $\fE(\cC)$ of {\em external} points, lying on two tangents to $\cC$
and the set $\fI(\cC)$ of {\em internal} points, lying on no tangent to $\cC$.

\begin{theorem}\label{th_3}
{\rm\cite{s}}
 Let $\cC: f(x,y,z)=0$ be a irreducible conic of $\PG(2,q)$, $q$ odd. Then a point $(x,y,z)$ is in  $\fE(\cC)$ if and only if   $-\det(A(\cC)) \cdot f(x,y,z)$  is a non-zero square in $
\GF(q)$.
\end{theorem}

Fix an irreducible conic $\cC$ in $\PG(2,q)$. In \cite{afkl}, irreducible
conics such that the points not in $\cC$ are all in $\fI(\cC)$ are described.
More precisely, the following theorem is proved.

\begin{theorem}\label{th_1}{\rm\cite{afkl}}
Let $\cC$ and $\cD$ be two irreducible conics of $\PG(2,q),\, q$ odd, 
$q\ge 17$, such that  $\cD\setminus\cC$ has empty intersection with $\fE(\cC)$. Then the  points of $\cC\setminus\cD$ consists entirely of internal points of $\cD$ and one of the following cases occur:
\begin{enumerate}[\rm(i)]

\item $\cC\cap \cD=\{P,Q\}$, $\cC$ and $\cD$ being two conics of a
bitangent pencil at $P$ and at $Q;$

\item $\cC \cap \cD =\emptyset$, $\cC$ and $\cD$ being two conics of a
bitangent pencil at $P$ and at $Q$, the two common points of $\cC$ and $\cD$
in the quadratic extension $\PG(2,q^2)$ of $\PG(2,q);$

\item $\cC\cap \cD=\{P\}$, $\cC$ and $\cD$ being two conics of a
hyperosculating pencil at $P$.
\end{enumerate}
\end{theorem}

It is worth pointing out that all the above pencils contain a conic of rank
1.

 The stabilizer of $\cC$ in the group $\PGL(3,q)$ of the
 linear collineations of $\PG(2,q)$ has three orbits on points of $\PG(2,q)$,
 namely, $\cC$ itself, $\fE(\cC)$ and $\fI(\cC)$. Dually, there are three
 orbits on lines, namely, the tangent lines to $\cC$, the secant lines to
 $\cC$ and the external lines to $\cC$. Since every line of $\PG(2,q)$ can be
 viewed as a conic of rank 1, we can fix a projective frame such that the
 conic $\cC$ and the pencils in Theorem \ref{th_1} have the following forms.
\begin{enumerate}[\rm(i)]
\item $\cC$ is the hyperbola $2xy=z^2$ and the pencil consists of the conics in the family $2xy=kz^2$, $k\in\GF(q)$, plus the repeated line $z^2=0$. The points $P$ and $Q$ are the points at infinity of $\cC$.
\item $\cC$ is the circle $x^2-\alpha y^2=z^2$ where $\alpha$ is a fixed non-square of $\GF(q)$ and the pencil consists of the conics in the family $x^2-\alpha y^2=kz^2$, $k\in\GF(q)$, plus the repeated line $z^2=0$. The points $P$ and $Q$ are the points at infinity of $\cC$ in $\PG(2,q^2)$.
\item $\cC$ is the parabola $2yz=x^2$ and the pencil consists of the conics in the family $2yz=x^2+kz^2$, $k\in\GF(q)$, plus the repeated line $z^2=0$. The point $P$ is the point at infinity of $\cC$.
\end{enumerate}

Assume now that $\cC$ is contained in $\cU$. Since every tangent to $\cC$ is
also a tangent to $\cU$ we see that $\cU\setminus \cC$ is contained in
$\fI(\cC)$.

In what follows we will use the representation of conics of $\PG(2,q)$ as
points of $\PG(5,q)$. We also recall some relevant properties of the Veronese
surface of $\PG(5,q)$. For a fuller treatment we refer the reader to
\cite[Chapter 25]{ht}.

If the 5-dimensional projective space $\PG(5,q)$ is equipped with the
homogeneous coordinates $(a_{11},a_{22}, a_{33}, a_{12}, a_{13},a_{23})$, the
conic $\cC$ with equation (\ref{eq_1}) defines the point
$P(\cC)=(a_{11},a_{22}, a_{33}, a_{12}, a_{13},a_{23})$ of $\PG(5,q)$, and
conversely. Under this 1-1 correspondence, the set of singular conics defines
the hypersurface with equation $\det(A)=0$ of $\PG(5,q)$, where $A$ is the
matrix associated with the generic conic $\cC$, and the set of rank 1 conics
defines the Veronese surface 
\[
\cV=\{(a^2,b^2,c^2,ab, ac,bc):a,b,c\in\GF(q),
(a,b,c)\neq (0,0,0)\}.
\] 
It is also easy to check that the representations in
$\PG(5,q)$ of the pencils of conics of type (i), (ii), (iii), are lines
intersecting the Veronese surface $\cV$ at $P=(0,0,1,0,0,0)$. Further, every
conic $\cC$ with rank $>1$ determines the cone $\Gamma(\cC)$ projecting $\cV$
from $P(\cC)$.

If $\cD$ is a second conic in $\cU$, then $\cD\setminus\cC$ has empty
intersection with $\fE(\cC)$. This implies that the pencil determined by $\cC$
and $\cD$ in $\PG(2,q^2)$ is one of those described in Theorem \ref{th_1}. We
also observe the symmetric relationship between the conics $\cC$ and $\cD$: if
all points of $\cD\setminus\cC$ are in $\fI(\cC)$ then all points of
$\cC\setminus\cD$ are in $\fI(\cD)$.

It is clear that the cones  $\Gamma(\cC)$ and  $\Gamma(\cD)$ share the line $P(\cC)P(\cD)$ and $\cV$.
By Theorem \ref{th_1},  for every other conic $\cE$  contained in $\cU$  the point $P(\cE)$ is contained in the intersection of the cones $\Gamma(\cC)$ and $\Gamma(\cD)$.

Since  $\cU$ does not contain lines of $\PG(2,q^2)$, then no point of $\cV$ represents a conic contained in $\cU$.
Hence we are reduced to studying which point in $(\Gamma(\cC)\cap\Gamma(\cD))\setminus \cV$ represents a conic in $\cU$. We will do this by considering the above three  cases for the pencil defined by $\cC$ and $\cD$.

{\bf Case 1.} {\em $\cC:2xy=z^2$ and $\cD:2xy=kz^2$ for some
$k\in\GF(q^2)\setminus\{0,1\}$}.

The intersection points between $\cC$ and $\cD$ are $P=(1,0,0)$ and
$Q=(0,1,0)$. Further, we have $\det(A(\cC))=1$ and $\det(A(\cD))=k$. By
Theorem \ref{th_3}, every point of $\cD\setminus\cC$ is in $\fI(\cC)$ if and
only if $k-1$ is a non-square in $\GF(q^2)$. By the symmetric relationship
between the conics $\cC$ and $\cD$, we have that $k(k-1)$ is a non-square of
$\GF(q^2)$. Hence, $k$ is a non-zero square of $\GF(q^2)$.

We first consider the irreducible conics of the pencil defined by $\cC$ and
$\cD$. If $\cE: 2xy=hz^2$, $h\neq 1,k,$ is contained in $\cU$, then $h$ is a
non-zero square in $\GF(q^2)$ and $h-1$, $h-k$ are non-squares in $\GF(q^2)$.
Hence such a set of conics determines a subset $X$ of $\GF(q^2)$ such that
$1\in X$, all elements of $X$ are non-zero squares and for any $h,k\in X$,
$h-k$ is a non-square.

\begin{lemma} \label{lem_1} 
Let $X$ be a subset of $\GF(q^2)$ of non-zero squares with the property that
the difference of any two elements is always a non-square. Then $X$ has at
most $(q+1)/2$ elements.
\end{lemma}
\begin{proof}
As usual we represent $\GF(q^2)$ as the affine plane $\AG(2,q)$. The lines of
this plane are subsets of $\GF(q^2)$ with the property that the difference of
two elements is either always a square, or always a non-square, depending only
on slope of the line. Thus the lines are partitioned into two classes, square
type $S$ and non-square type $N$. Through each point of $\AG(2,q)$ there pass
$(q+1)/2$ lines of $S$ and $(q+1)/2$ lines of $N$. Hence, on an arbitrary line
$L$ of $S$ not passing through the origin $O$, there are $(q+1)/2$ non-squares,
since the line parallel to $L$ containing the origin is also in $S$.

Let $A$ and $B$ two distinct points of $X$ collinear with the origin. Then
$A-B$ is always a square, a contradiction. This implies that on each line of
type $S$ on $O$ there is at most one point of $X$. Then $X$ contains at most
$(q+1)/2$ points.
\end{proof}

A consequence of this lemma is that the conics of the pencil defined by $\cC$
and $\cD$ cover at most $2+(q^2-1)(q+1)/2$ points of $\cU$. Since $q>3$, in
order to cover the remaining points of $\cU$ we need more then one conic not
in the pencil defined by $\cC$ and $\cD$. So we investigate $\Gamma(\cC)\cap
\Gamma(\cD) \setminus (P(\cC) P(\cD) \cup \cV)$.

It is easily seen that the points of $\Gamma(\cC)$ and $\Gamma(\cD)$ not in
the surface $\cV$ are 
\begin{eqnarray*}
&&\hspace*{-8mm}\{(sa^2,sb^2,1+sc^2,-1+sab,sac,sbc):a,b,c,s\in\GF(q^2),
(a,b,c)\neq(0,0,0)\},\\
&&\hspace*{-8mm}\{(ta'^2,tb'^2,k+tc'^2,-1+ta'b',ta'c',tb'c'):a',b',c',t\in\GF(q^2),
(a',b',c')\neq(0,0,0)\}.
\end{eqnarray*}
It is worth pointing out that the points of the line
$P(\cC)P(\cD)$ are those for which $(a,b,c)=(a',b',c')=(0,0,1)$. Further, we
get $P(\cC)$ and $P(\cD)$ also for $s=0$ and $t=0$, respectively. In the
following we assume $(a,b,c)\neq(0,0,1)\neq(a',b',c')$ and $st\neq0$.

Then the points in $(\Gamma(\cC)\cap\Gamma(\cD))\setminus(P(\cC)P(\cD)\cup\cV)$ satisfy the following equations:

\begin{equation}\label{eq_2}
\begin{split}
sa^2  &=\rho ta'^2  \\
 sb^2 &=\rho tb'^2  \\
1+sc^2&=\rho(k+ tc'^2) \\
-1+sab&=\rho(-1+ ta'b')\\
sac   &=\rho ta'c'\\
sbc   &=\rho tb'c'
\end{split}
\end{equation}
{}for some $\rho\in\GF(q^2)^*$.

{}First we consider $abc\neq0$. {}From  Equations (\ref{eq_2}), we get 
$a'b'c' \ne 0$ and
\[
\frac{s}{t} = \frac{\rho a'^2}{a^2} =
\frac{\rho a'c'}{ac}= \frac{\rho b'c'}{bc},\ {\rm i.e.}\
\frac{a'}{a} = \frac{b'}{b} =\frac{c'}{c}.
\]

 This implies that the generator with base point $(a,b,c)$ of $\Gamma(\cC)$ meets $\Gamma(\cD)$ on $\cV$. Hence there are no points in   $(\Gamma(\cC)\cap\Gamma(\cD))\setminus(P(\cC)P(\cD)\cup\cV)$  with $abc\neq0$.

Now suppose that $a=0$ and $bc\neq0$. Equations (\ref{eq_2}) reduce to
\begin{equation}
\begin{split}
0  &= a'^2  \\
 sb^2 &=\rho tb'^2  \\
1+sc^2&=\rho(k+ tc'^2) \\
-1&=\rho(-1+ ta'b')\\
0   &= a'c'\\
sbc   &=\rho tb'c'.
\end{split}
\end{equation}
It follows immediately that $a'=0$ and $b'c'\neq0$; hence we get the same
conclusion as before.

The same reasoning applies to the case $b=0$ and $ac\neq0$.

If $c=0$ and $ab\neq0$, Equations (\ref{eq_2}) reduce to
\begin{equation}\label{eq_3}
\begin{split}
sa^2  &=\rho ta'^2  \\
 sb^2 &=\rho tb'^2  \\
1&=\rho(k+ tc'^2) \\
-1+sab&=\rho(-1+ ta'b')\\
0   &= a'c'\\
0   &= b'c'.
\end{split}
\end{equation}
As $ab\neq0$ we have $a'b'\neq0$, $c'=0$ and $\rho=k^{-1}$. We can assume that $a=1=a'$. Thus Equations (\ref{eq_3}) reduce to
\begin{equation}\label{eq_5}
\begin{split}
 s&=k^{-1} t  \\
b^2&=b'^2 \\
-1+sb&=-k^{-1}+k^{-1}tb'.
\end{split}
\end{equation}

Since $b^2=b'^2$, we have either $b=b'$ or $b=-b'$. If $b=b'$ we get $k=1$, a contradiction. Hence $b=-b'$. Then $t=(1-k)/2b'$ and it is easy to check that the cones $\Gamma(\cC)$ and $\Gamma(\cD)$ share the points $P_{\cE_{b'}}=(1-k,(1-k)b'^2,2kb',-(k+1)b',0,0))$, $b'\in\GF(q^2)$. We note that the conic  \begin{equation}\label{eq_4}
\cE_{b'}:(1-k)x^2+(1-k)b'^{2}y^2+2kb'z^2-2(1+k)b'xy=0,
\end{equation}
 has rank 3 for all $b'$.

We proceed by considering separately the cases $b'$ a non-square and $b'$ a
non-zero square.

Let $b'$ be a non-square of $\GF(q^2)$. As $k(k-1)$ is a non-square, we see
that the line $x=0$ intersects $\cE_{b'}$ in $(0,\bar y,1)$, with $\bar
y=\sqrt{\frac{2k}{(k-1)b'}}$. By Theorem \ref{th_3}, we have that $(0,\bar
y,1)$ is in $\fE(\cC)$. Then $\cE_{b'}$ cannot be contained in $\cU$.

We now turn to the case $b'$ a non-zero square. Assume that $\cE_{b'}$ is
contained in $\cU$. As $q>3$, $\cU$ contains another conic $\tilde\cE$. By
applying the same reasoning to $\tilde\cE$, we get
$$
\tilde\cE=\cE_{b''}:(1-k)x^2+(1-k)b''^{2}y^2+2k b''z^2-2(1+k)b''xy=0.
$$ 
Since $\cE_{b'}$ and $\cE_{b''}$ are contained in $\cU$, they define one of
the pencils in Theorem \ref{th_1}. This implies that the line defined by
$P(\cE_{b'})$ and $P(\cE_{b''})$ in $\PG(5,q)$ should intersect the surface
$\cV$. But we will see that this is not the case.

To simplify calculations, we apply the collineation
\begin{equation}\label{eq_7}
\sigma:\left\{
 \begin{array}{rl}
x' &=x\\
y' &=b' y  \\
z' &=\sqrt{b'}z
\end{array}\right.
\end{equation}
of $\PG(2,q^2)$.

Then $\sigma$ takes $\cE_{b'}$ to $\cE_{1}$  and $\cE_{b''}$ to
and $\cE_{\beta}$, with $\beta=b''/b'\neq1$. The line of $\PG(5,q)$ 
defined by 
\begin{eqnarray*}
P(\cE_{1}) & = &(1-k,1-k,2k,-(1+k),0,0)\\
P(\cE_{\beta})& = &(1-k,(1-k)\beta^2,2k\beta,-\beta(1+k),0,0)
\end{eqnarray*} 
intersects $\cV$ if and only if
\begin{equation} \label{eq_6}
\begin{split}
 (1-k)(1+s)  &=\rho l^2  \\
  (1-k)(1+s\beta^2) &=\rho m^2  \\
1+s\beta &=\rho n^2 \\
 -(1+k)(1+s\beta)&=\rho lm\\
0   &= ln\\
0   &= mn
\end{split}
\end{equation}
{}for some $l,m,n\in\GF(q^2)$, $(l,m,n)\neq(0,0,0)$, and $s,\rho\in\GF(q^2)^*$; we recall that $k\neq1$.
In the following we use Equations (\ref{eq_6}).

Assume $n=0$ and $lm\neq0$. Then $1+s\beta=0$ and this implies that $\rho
lm=0$, a contradiction. Assume $n=0=l$. Without loss of generality we may
assume $m=1$. Then $s=-1$ and $\beta=1$, a contradiction. Assume $n=0=m$.
Without loss of generality we may assume $l=1$. Then $1+s\beta=0$. Hence
$s\beta=-1$. These forces $\beta=1$, a contradiction. Assume $l=0=m$. Without
loss of generality we may assume $n=1$. Then $s=-1$ and $\beta=-1$. This
forces $k=-1$ which contradicts the fact that $k-1$ has to be a non-square in
$\GF(q^2)$. This proves that $\cU$ cannot contain the conic $\cE_{b''}$, a
contradiction.

We leave it to the reader to verify that, when $(a,b,c)$ is either $(0,1,0)$
or $(1,0,0)$, there are no points in
$(\Gamma(\cC)\cap\Gamma(\cD))\setminus(P(\cC)P(\cD)\cup\cV)$.

\comment{

\item $a=b=0, c=1$ or $a=c=0, b=1$ or $b=c=0, a=1$.

If $a=b=0$ it follows  that $a'=0, b'=0$ hence the generators with base point $Q(s,0,0,1)$ and $Q'(t,0,0,1)$ meet just in ${\cal V}$.

If $a=c=0$ it follows $a'=0$ and $b'c'=0$. Since $sb^2=\rho(tb'^2)$ from $b\ne 0$ we have $b'\ne 0$ and hence $c'=0$.
So again  the generators with base point $Q(s,0,1,0)$ and $Q'(t,0,1,0)$ meet just in ${\cal V}$.

If $b=c=0$ and $a=1$ then $b'=0$ and $a'c'=0$ but since $a\ne 0$ it follows $a'\ne 0$ and hence $c'=0$ and again the generators meet just in ${\cal V}$.
}

Hence, we have proved that there is at most one conic $\cE$ not in the pencil
defined by $\cC$ and $\cD$ that can be contained in $\cU$. From Lemma
\ref{lem_1}, we get that does not exist a unital $\cU$ which is union of
irreducible conics with two conics defining a pencil of type (i).

{\bf Case 2.} {\em  $\cC:x^2-\alpha y^2=z^2$ and $\cD:x^2-\alpha y^2=kz^2$ for a fixed non-square  $\alpha\in\GF(q^2)$ and some $k\in\GF(q^2)\setminus\{0,1\}$}.

The conics $\cC$ and $\cD$ have empty intersection in $\PG(2,q^2)$. Further, we have $\det(A(\cC))=\alpha$ and $\det(A(\cD))=\alpha k$. By Theorem \ref{th_3}, every point of $\cD\setminus\cC$ is in $\fI(\cC)$ if and only if $k-1$ is a non-zero square in $\GF(q^2)$. By the symmetric relationship between the conics $\cC$ and $\cD$, we have that $k(k-1)$ is non-zero square in $\GF(q^2)$. Hence $k$ must be a non-zero square of $\GF(q^2)$.

We now proceed similarly to the previous case. We first consider the
irreducible conics of the pencil defined by $\cC$ and $\cD$. We point out that
these conics are  disjoint in $\PG(2,q^2)$. As $q^2+1$ does not divide
$q^3+1$ we see that $\cU$ must contain a further conic not in the pencil. So
we investigate $\Gamma(\cC)\cap \Gamma(\cD) \setminus (P(\cC) P(\cD) \cup
\cV)$.


It is easily seen that the points of $\Gamma(\cC)$ and $\Gamma(\cD)$ not in
the surface $\cV$ are
\begin{eqnarray*}
&&\hspace*{-7mm}\{(1+sa^2,-\alpha+sb^2,-1+sc^2,sab,sac,sbc): a,b,c,s\in\GF(q^2),
(a,b,c)\neq(0,0,0)\},\\
&&\hspace*{-7mm}\left\{(1+ta'^2,-\alpha+tb'^2,-k+tc'^2,ta'b',ta'c',tb'c'): 
a',b',c',t\in\GF(q^2),\right.\\
&&\hspace*{10cm}\left.(a',b',c')\neq(0,0,0)\right\}.
\end{eqnarray*} 
As in the previous case, we have
$(a,b,c)\neq(0,0,1)\neq(a',b',c')$ and $st\neq0$.

Then, the points in $\Gamma(\cC)\cap\Gamma(\cD)\setminus(P(\cC)P(\cD)\cup\cV)$
satisfy the following equations:
\begin{equation}\label{eq_8}
\begin{split}
1+sa^2  &=\rho (1+ta'^2)  \\
 -\alpha+sb^2 &=\rho (-\alpha+tb'^2)  \\
-1+sc^2&=\rho(-k+ tc'^2) \\
sab&=\rho ta'b'\\
sac   &=\rho ta'c'\\
sbc   &=\rho tb'c'
\end{split}
\end{equation}
{}for some $\rho\in\GF(q^2)^*$.

 {}First we consider $abc\neq0$. {}As in Case 1, from the above equations, we
 get that there are no points in
 $\Gamma(\cC)\cap\Gamma(\cD)\setminus(P(\cC)P(\cD)\cup\cV)$ with $abc\neq0$.

Now suppose that $a=0$ and $bc\neq0$. Equations (\ref{eq_8}) reduce to
\begin{equation*}
\begin{split}
1  &=\rho (1+ta'^2)  \\
 -\alpha+sb^2 &=\rho (-\alpha+tb'^2)  \\
-1+sc^2&=\rho(-k+ tc'^2) \\
0&=a'b'\\
0   &= a'c'\\
sbc   &= \rho tb'c'.
\end{split}
\end{equation*}

It follows immediately that $a'=0$ and $b'c'\neq0$; hence we get the same
conclusion as before. The same reasoning applies to the case $b=0$ and
$ac\neq0$.

If $c=0$ and $ab\neq0$, Equations (\ref{eq_8}) reduce to
\begin{equation}\label{eq_9}
\begin{split}
1+sa^2  &=\rho (1+ta'^2)  \\
 -\alpha+sb^2 &=\rho (-\alpha+tb'^2)  \\
-1&=\rho(-k+ tc'^2) \\
sab&=\rho ta'b'\\
0   &= a'c'\\
0   &= b'c'.
\end{split}
\end{equation}
As $ab\neq0$ we have $a'b'\neq0$, $c'=0$ and $\rho=k^{-1}$. We can assume that $a=1=a'$. Thus Equations (\ref{eq_9}) reduce to
\begin{equation}\label{eq_10}
\begin{split}
1+s  &=k^{-1} (1+t)  \\
 -\alpha+sb^2 &=k^{-1} (-\alpha+tb'^2)  \\
sb&=k^{-1} tb'.
\end{split}
\end{equation}

{}From the first and third equation of  (\ref{eq_10}) we get
\[
\begin{split}
t&=k(1+s)-1\\
b'&=\frac{sbk}{t}.
\end{split}
\]
By substituting these expressions into the second equation of  $(\ref{eq_10}),$  we get
\begin{equation}\label{eq_11}
s=\frac{\alpha(1-k)}{k(\alpha-b^2)}.
\end{equation}

By substituting $a=1$, $c=0$, $\rho=k^{-1}$ and (\ref{eq_11}) into Equations
(\ref{eq_9}), we get that
$\Gamma(\cC)\cap\Gamma(\cD)\setminus(P(\cC)P(\cD)\cup\cV)$ consists of the
points 
\[
P(\cE_{b,k})=(\alpha-kb^2,\alpha(b^2-\alpha k),k(b^2-\alpha),\alpha
b(1-k),0,0),
\] 
with $b\in\GF(q^2)^*$.

In order for $\cE_{b,k}$ to be a conic in $\cU$, the sets $\cC\setminus\cE_{b,k}$ and
$\cD\setminus\cE_{b,k}$ should be both contained in $\fI(\cE_{b,k})$. By using
Theorem \ref{th_3}, with straightforward calculations we obtain that the point
$(1,0,1)$ of $\cC$ is in $\fI(\cE_{b,k})$ if and only if $b^2-\alpha$ is a
non-square of $\GF(q^2)$ and the point $(\sqrt k,0,1)$ of $\cD$ is in
$\fI(\cE_{b,k})$ if and only if $b^2-\alpha$ is a non-zero square of
$\GF(q^2)$, a contradiction. Thus we can conclude that no conic $\cE_{b,k}$ is
contained in $\cU$.

Assume $a=c=0$, so we can suppose $b=1$. Equations (\ref{eq_8}) reduce to
\begin{equation}\label{eq_12}
\begin{split}
1  &=\rho (1+ta'^2)  \\
 -\alpha+s &=\rho (-\alpha+tb'^2)  \\
-1&=\rho(-k+ tc'^2) \\
0 &= a'b'\\
0   &= a'c'\\
0   &= b'c'.
\end{split}
\end{equation}
If $a'=0$ it is easily seen that there are no points in  
$\Gamma(\cC)\cap \Gamma(\cD) \setminus (P(\cC) P(\cD) \cup \cV)$. 
Hence $a'\ne 0$ and  $b'=0=c'$. Equations  (\ref{eq_12}) reduce to
\begin{equation*}\label{eq_13}
\begin{split}
1  &=\rho (1+t)  \\
 -\alpha+s &=-\rho\alpha  \\
-1&=-\rho k.
\end{split}
\end{equation*}
Hence, $\rho=k^{-1}$, $k=1+t$ and $s=\alpha (1-k^{-1})$ and we get the unique
common point $P(\cE)=(k,-\alpha,-k,0,0,0)$. In order for $\cE$ to be a conic
in $\cU$, we should have that the line $P(\cE)P(\cG)$ intersects the surface
$\cV$ in exactly one point, for every conic $\cG$ of the pencil defined by
$\cC$ and $\cD$ and contained in $\cU$. But this happens if and only if $\cG$
coincides with either  $\cC$ or $\cD$. Since $q>3$, the conics $\cC$, $\cD$ and
$\cE$ don't cover all points of $\cU$.

Assume $b=c=0$; so we can suppose $a=1$. Equations (\ref{eq_8}) reduce to
\begin{equation}\label{eq_14}
\begin{split}
1 +s &=\rho (1+ta'^2)  \\
 -\alpha &=\rho (-\alpha+tb'^2)  \\
-1&=\rho(-k+ tc'^2) \\
0 &= a'b'\\
0   &= a'c'\\
0   &= b'c'.
\end{split}
\end{equation}
If $b'=0$ it is easily seen that there are no points in  $\Gamma(\cC)\cap \Gamma(\cD) \setminus (P(\cC) P(\cD) \cup \cV)$. Hence $b'\ne 0$ and  $a'=0=c'$. Equations  (\ref{eq_14}) reduce to
\begin{equation*}
\begin{split}
1+s  &=\rho  \\
 -\alpha &=\rho (-\alpha+t)  \\
-1&=-\rho k.
\end{split}
\end{equation*}

Analysis similar to the above case shows that the conics contained in $\cU$
are precisely $\cC$, $\cD$ and $\cE:x^2-\alpha k y^2 =kz^2$, a contradiction.
Finally, we conclude that there does not exist a unital $\cU$ which union of
irreducible conics with two conics defining a pencil of type (ii).

{\bf Case 3.} {\em $\cC:2yz=x^2$ and $\cD:2yz=x^2+kz^2$ for some
$k\in\GF(q^2)\setminus\{0\}$}.

The intersection between $\cC$ and $\cD$ is the point $P(0,1,0)$. Further, we
have $\det(A(\cC))=\det(A(\cD))=-1$. By Theorem \ref{th_3}, every point of
$\cD\setminus\cC$ is in $\fI(\cC)$ if and only if $k$ is a non-square in
$\GF(q^2)$.

We first consider the irreducible conics in the pencil defined by $\cC$ and
$\cD$. If a conic $\cE: 2yz=x^2+hz^2$, $h\ne 1,k$, is contained in $\cU$ then
$h$, $h-k$ are non-square in $\GF(q^2)$.

Hence such a set of conics determines a subset $X$ of $\GF(q^2)$ such that
$1\in X$, all elements of $X$ are non-squares and for any $h,k\in X$, $h-k$ is
a non-square. To obtain a unital $X$ must have size $q$.

 \begin{lemma} \label{lem_2} 
 {\rm\cite{bl} }
 Let $X$ be a subset of $\GF(q^2)$ of non-squares such that the
difference of any two elements is always a non-square. If $|X|=q$, then $X=t
\GF(q)$ for some non-square $t\in\GF(q^2)$.
\end{lemma}

It follows that such a set $X$ gives a Buekenhout--Metz unital of BEHS-type.

In order to investigate if there are further unitals union of conics we need,
also in this case, to study $\Gamma(\cC)\cap \Gamma(\cD) \setminus (P(\cC)
P(\cD) \cup \cV)$:
\begin{equation}\label{eq_30}
\begin{split}
1+sa^2  &=\rho (1+ta'^2)  \\
sb^2 &=\rho tb'^2  \\
sc^2&=\rho(-k+ tc'^2) \\
sab&=\rho ta'b'\\
sac   &=\rho ta'c'\\
-1+sbc   &=\rho (-1+tb'c')
\end{split}
\end{equation}
for some $\rho\in\GF(q)^*$. 
Also in this case we have $(a,b,c)\neq(0,0,1)\neq(a',b',c')$ and $st\neq0$.

If either $abc\ne 0$ or $a=0$  or $b=0$, it is easy to check that there are no points in
$\Gamma(\cC)\cap\Gamma(\cD)\setminus(P(\cC)P(\cD)\cup\cV)$.

If $c=0$ and $ab\ne 0$, then  Equations (\ref{eq_30}) become
\begin{equation*}
\begin{split}
1+sa^2 &= \rho(1+ta'^2)  \\
 sb^2 &=\rho t b'^2  \\
0 &=-k+tc'^2 \\
sab &= \rho ta'b' \\
0 &=a'c' \\
-1 &=\rho(-1+tb'c').
\end{split}
\end{equation*}
If $a'=0$, then $ab=0$, a contradiction. If $c'=0$, then $k=0$, a
 contradiction.
If $(a,b,c)$ iss either $(0,1,0)$ or $(1,0,0)$, we leave to the reader to
verify that there are no points in $\Gamma(\cC)\cap \Gamma(\cD) \setminus
(P(\cC) P(\cD) \cup \cV)$.

 This concludes the proof of Theorem \ref{th_2}.

\end{document}